\documentclass[a4paper]{amsart}
\usepackage[cmtip,all]{xy}

\newtheorem{theorem}{Theorem}[section]
\newtheorem{lemma}[theorem]{Lemma}
\newtheorem{proposition}[theorem]{Proposition}

\theoremstyle{definition}
\newtheorem{definition}[theorem]{Definition}

\theoremstyle{remark}
\newtheorem{remark}[theorem]{Remark}

\numberwithin{equation}{section}

\DeclareMathOperator{\Pic}{Pic}
\DeclareMathOperator{\Tot}{Tot}
\DeclareMathOperator{\dlog}{dlog}
\DeclareMathOperator{\cls}{cls}
\DeclareMathOperator{\NS}{NS}
\DeclareMathOperator{\red}{red}
\DeclareMathOperator{\Spec}{Spec}
\DeclareMathOperator{\image}{im}
\DeclareMathOperator{\cosk}{cosk}
\DeclareMathOperator{\id}{id}
\DeclareMathOperator{\coker}{coker}
\DeclareMathOperator{\et}{\text{\'et}}
\DeclareMathOperator{\ab}{ab}
\DeclareMathOperator{\cont}{cont}
\DeclareMathOperator{\Der}{Der}

\DeclareMathOperator{\Char}{char}
\newcommand{\N}{\mathbb{N}}
\newcommand{\F}{\mathbb{F}}
\newcommand{\A}{\mathbb{A}}
\newcommand{\G}{\mathbb{G}}
\newcommand{\Z}{\mathbb{Z}}
\newcommand{\C}{\mathbb{C}}
\begin{document}

\title[The Picard Group of Simply Connected Varieties]{The Picard Group of Simply
Connected Regular Varieties and Stratified Line Bundles}


\author{Lars Kindler}
\thanks{This work was supported by the Sonderforschungsbereich/Transregio 45 ``Periods,
moduli spaces and the arithmetic of algebraic varieties'' of the DFG} 

\address{%
Universit\"at Duisburg-Essen, Mathematik, 45117 Essen, Germany}
\email{lars.kindler@uni-due.de}
\date{April 12, 2011}




\begin{abstract}
We prove that the Picard group of a regular simply connected variety over an algebraically closed field 
of arbitrary characteristic is finitely generated. The main difficulty to overcome is the
unavailability of resolution of singularities. From this we deduce that in positive
characteristic there exist no nontrivial stratified line bundles on such a variety, and we
present a complex analog.
\end{abstract}

\maketitle
\section{Introduction}
Let $k$ be an algebraically closed field of characteristic $p\geq 0$, and $U$ a regular
connected $k$-scheme of finite type. Assume that $U$ is simply
connected, i.e.~$\pi_1^{\et}(U,\bar{u})=1$ for some geometric point $\bar{u}$ of $U$. If
there is a dominant open immersion of $k$-schemes $\iota:U\hookrightarrow X$, for some
regular proper
$k$-scheme $X$, then $\Pic U$ is finitely generated. In fact, by the regularity
assumptions we have  surjections $\Pic X\rightarrow \Pic U$ and
$\pi_1(U,\bar{u})\rightarrow \pi_1(X,\iota\bar{u})$, so $X$ is also simply connected and
it suffices to show that $\Pic X$ is finitely generated. By
the properness of $X$ 
the relative Picard functor for $X/k$  is representable by a $k$-group scheme $\Pic_{X/k}$ locally of finite
type, and the connected component of the origin with its reduced structure
$\Pic^{0,\red}_{X/k}$ is an abelian variety. The Kummer sequence in \'etale cohomology
then shows that
$\Pic^{0,\red}_{X/k}[\ell^n]=\hom_{\cont}(\pi_1(X,\iota\bar{u}),\Z/\ell^n\Z)=1$ for any
prime $\ell\neq p$. This
implies that $\Pic^0_{X/k}(k)$ is trivial, as for an abelian variety $A$  of dimension $g$ over $k$, we know
that $A[\ell^n]\cong (\Z/\ell^n\Z)^{2g}$. Hence the Picard group $\Pic(X)=\Pic_{X/k}(k)$ is equal to the N\'eron-Severi group
$\NS(X):=\Pic_{X/k}(k)/\Pic_{X/k}^{0}(k)$, which is finitely generated.

The first main result of this note is a generalization of the above fact, proven in
Section \ref{section:simplyconnectedvarieties}.
\begin{theorem}\label{picfg}Let $k$ be an algebraically closed field of characteristic
	$p\geq 0$. If $U$ is a connected, regular, separated $k$-scheme of finite type, and if the
	maximal abelian pro-$\ell$ quotient $\pi_1(U)^{ab,(\ell)}$ is trivial for some
	$\ell\neq p$, then $\Pic
	U$ is finitely generated.
\end{theorem}

Note that it is not known in general, whether $U$ as in the theorem can be embedded into a
regular proper $k$-scheme, because resolution of singularities is not known to hold. To
circumvent this, we use de Jong's theory of alterations and simplicial techniques. More
precisely, in Section \ref{simplicialsec} we study the simplicial Picard group and the
simplicial N\'eron-Severi group of a simplicial scheme $X_{\bullet}$, and prove finiteness
statements similar to the nonsimplicial case. This relies on results of
\cite{Ramachandran/Motives} and \cite{Srinivas/Motives}. We apply these statements in Section
\ref{section:simplyconnectedvarieties} to prove the theorem.

As an application, we study in Section \ref{section:stratified} stratified line bundles on
$U$, i.e.~line bundles coming
with a $\mathcal{D}_{U/k}$-action, where $\mathcal{D}_{U/k}$ is the sheaf of differential
operators of $X$. Extending an argument from \cite{Esnault/Giesecker},  we prove the
following statement.
\begin{theorem}\label{thm:nostratifiedlinebundles}
	Let $k$ be an algebraically closed field of characteristic $p> 0$.
	If $U$ is a regular, connected $k$-scheme of finite type such that the
	maximal abelian pro-prime-to-$p$ quotient
	$\pi_1(U)^{\ab,(p')}$ of $\pi_1(U,\bar{u})$ is trivial, then
	every stratified line bundle on $U$ is trivial.
\end{theorem}

In \cite{Esnault/Giesecker}, this theorem is proven for
stratified bundles of arbitrary rank, under the additional assumptions that $U$ is
\emph{projective} and $\pi_1(U)=1$.

Proving this theorem was the original motivation to study Theorem \ref{picfg},
and it originated out of work for my thesis, in which possible extensions of the results of
\cite{Esnault/Giesecker} to non-projective varieties are studied. 

Finally, in Section \ref{section:complex} we reproduce an argument of H\'el\`ene Esnault,
showing that a complex analog of Theorem
\ref{thm:nostratifiedlinebundles} is ``not quite'' correct. More precisely, we show that
on a simply connected, regular, complex variety $U$, every stratified line bundle is of the
form $(\mathcal{O}_U,d+\omega)$, with $\omega$ a closed $1$-form. The reason for this
discrepancy is
that in positive characteristic, simply connectedness is a ``stronger'' condition than
in characteristic $0$, see Remark \ref{rem:end}.

{\it Acknowledgment:} I happily acknowledge,
and express tremendous gratitude
for the great influence and generous help of my adviser H\'el\`ene Esnault.

\section{Simplicial Picard Groups}\label{simplicialsec}
For background on the simplicial techniques used in this section, we refer to
\cite{HodgeIII}.

In \cite[4.1]{Srinivas/Motives} and \cite[3.1]{Ramachandran/Motives}, the
simplicial Picard group and the simplicial Picard functor are defined as follows:
\begin{definition}\label{definition:simplicialpic}
	Let $S$ be a scheme. If $\delta_k^i:X_i\rightarrow X_{i-1}$ denote the face maps of a simplicial $S$-scheme
$X_{\bullet}$, then $\Pic(X_\bullet)$ is defined to be the group of isomorphism classes of pairs $(L,\alpha)$ consisting of a
line bundle $L$ on $X_0$ and an isomorphism $\alpha:(\delta^1_0)^*L\rightarrow (\delta^1_1)^*L$ on
$X_1$, satisfying cocycle condition $(\delta_2^2)^*(\alpha)
(\delta^2_0)^*(\alpha)=(\delta_1^2)^*(\alpha)$ on $X_2$.  The
simplicial Picard functor $\Pic_{X_\bullet/S}$ is obtained by fpqc-sheafifying the functor
$T\mapsto
\Pic(X_\bullet\times_S T)$.
\end{definition}
It turns out that $\Pic(X_\bullet)$ is
canonically isomorphic to $\mathbb{H}^1(X_\bullet, \mathcal{O}^\times_{X_\bullet})$ and to the
group of isomorphism classes of invertible $\mathcal{O}_{X_\bullet}$-modules, see
\cite[A.3]{Srinivas/Motives}. We will use the following representability and finiteness
statements.
\begin{theorem}\label{representabilityandfiniteness}
	Let $k$ be an algebraically closed field, $X$ a proper $k$-scheme of
	finite type, and $X_\bullet$ a proper, simplicial $k$-scheme of finite
	type (which means that all the $X_n$ are proper, and of finite type over
	$k$).
\begin{enumerate}
	\item The relative Picard
		functor associated to $X\rightarrow \Spec k$ is representable by a separated commutative group scheme
		$\Pic_{X/k}$, locally of finite
		type over $k$, which is the disjoint union of open, quasi-projective
		subschemes, see \cite[Cor. XII.1.2]{SGA6}.
	\item\label{nsfg} The N\'eron-Severi group $\NS(X)=\Pic_{X/k}(k)/\Pic^0_{X/k}(k)$ is a finitely
		generated abelian group, see \cite[Thm. XIII.5.1]{SGA6}.
	\item If $X$ also \emph{normal}, then
		the connected component $\Pic_{X/k}^0$ of the origin is projective, see
		\cite[Thm. 5.4, Rem.
		5.6]{Kleiman/PicardScheme}, so passing to the reduced structure
		$\Pic_{X/k}^{0,\red}$ gives an abelian
		variety.
	\item The simplicial Picard
		functor is representable by a group scheme
		$\Pic_{X_\bullet/k}$, locally of finite type over $k$,
		see \cite[Thm. 3.2]{Ramachandran/Motives}.
	\item If $X_n$ is reduced for all $n$, and $X_0$ normal, then the connected
		component $\Pic_{X_\bullet/k}^{0,\red}$ of the origin is semi-abelian, see \cite[Cor. 3.5]{Ramachandran/Motives}.
\end{enumerate}
\end{theorem}

Statement \ref{nsfg} from above can be generalized to the simplicial situation. In the case
that $\Char(k)=0$, this is sketched in \cite[Sec.
3]{Barbieri/NeronSeveri}.
\begin{proposition}\label{simplicialnsfg}
	Let $k$ be an algebraically closed field. For a proper reduced simplicial $k$-scheme
	$X_\bullet$, with $X_n$ of finite type for all $n$ and $X_0$ normal, the \emph{simplicial N\'eron-Severi
	group} $\NS(X_\bullet):=\Pic_{X_\bullet/k}(k)/\Pic_{X_\bullet/k}^0(k)$ is finitely
	generated.
\end{proposition}
\begin{proof}
	Let $\tau:X_\bullet\rightarrow \Spec k$ denote the structure morphism.
	The spectral sequence (see e.g. \cite[(5.2.3.2)]{HodgeIII})
	\[E_1^{p,q}=H^q(X_p,\mathcal{O}_{X_p}^\times)\Longrightarrow
	\mathbb{H}^{p+q}(X_\bullet,\mathcal{O}_{X_\bullet}^\times)\]
	gives rise to the exact sequence
	\[ 0\longrightarrow E^{1,0}_{2}\longrightarrow
	\mathbb{H}^{1}(X_{\bullet},\mathcal{O}_{X_{\bullet}})\longrightarrow
	E_2^{0,1}\stackrel{d_2}{\longrightarrow} E^{2,0}_2\]
	and thus, as in \cite[p. 284]{Ramachandran/Motives}, after sheafifying we get an
	exact sequence of fpqc-sheaves (in fact group schemes by the representability of
	$\Pic_{X_\bullet/k}$)
	\[0\rightarrow T\rightarrow \Pic^{\red}_{X_\bullet/k}\rightarrow
	K\stackrel{d_2}{\rightarrow} W,\]
	where $K:=\ker( \delta_0^*- \delta_1^*:\Pic_{X_0/k}\rightarrow
	\Pic_{X_1/k})^{\red}$,\[T:=\frac{\ker(
	(\tau_1)_*\mathbb{G}_{m,X_1}\stackrel{\delta_0^*-\delta_1^*+\delta_2^*}{\longrightarrow}
	(\tau_2)_*\mathbb{G}_{m,X_2})}{\image(
	(\tau_0)_*\mathbb{G}_{m,X_0}\stackrel{\delta^*_0-\delta^*_1}{\longrightarrow}
	(\tau_1)_*\mathbb{G}_{m, X_1})},\]
	and $W$ is affine.
	The scheme $T$  is an affine $k$-scheme with
	finitely many connected components, and a $k$-torus as neutral component (compare \cite[top of
	p. 284]{Ramachandran/Motives}). This follows from the fact that $\tau_n$ is proper and $X_n$
	reduced, since this implies that for every $k$-scheme $S$ we have
	$\tau_{n,*}\G_{m,X_n}(S)=\mathcal{O}_{X_n\times_k S}^\times(X_n\times_k
	S)=\G_{m,k}(S)^{\pi_0(X_n)}$, see \cite[Prop. 7.8.6]{EGAIII}.

	As $X_0$ is normal, $\Pic_{X_0/k}^{0,\red}$ is an abelian variety, and hence so is
	$K^0$. Since $W$ is affine, any homomorphism $K^0\rightarrow W$ is trivial, so
	$K^0\subset \ker(d_2)^0$. But $\ker(d_2)^0\subset K^0$, so we have equality.
	Moreover, $\Pic_{X_\bullet/k}^{0,\red}$ maps surjectively to $K^0$. This shows
	that the kernel of the map 
	\[ \Pic_{X_\bullet/k}(k)/\Pic_{X_{\bullet}/k}^{0}(k)=\NS(X_\bullet)\rightarrow
	K(k)/K^0(k) \]
	is 
	\[\frac{T(k)\Pic^0_{X_{\bullet}/k}(k)}{\Pic^0_{X_\bullet/k}(k)},\]
	because, if $L$ maps to $M\in K^0(k)$, then there is some
	$L^0\in \Pic^0_{X_\bullet/k}(k)$ also mapping to $M$, and the difference is in
	$T(k)$.

	As $T$ has only finitely many connected components, this kernel is finite.

	The group $K(k)/K^0(k)$ maps to the group of connected components $\NS(X_0)$ of
	$\Pic_{X_0/k}$. The kernel of this map is $(\Pic^0_{X_0/k}(k)\cap K(k))/K^0(k)$,
	which is finite, as $\Pic^0_{X_0/k}\cap K$ has only finitely many connected
	components, and the subgroup $K^0\subset \Pic^0_{X_0/k}$ is the neutral component of
	$\Pic^0_{X_0/k}\cap K$. Hence $K(k)/K^0(k)$ is finitely generated, because
	$\NS(X_0)$ is finitely generated by Theorem \ref{representabilityandfiniteness}. This shows that
	$\NS(X_\bullet)$ is finitely generated.

\end{proof}

Recall that to an $X$-scheme $X_0$ one can associate an $X$-augmented simplicial scheme
$\cosk_0(X_0)_\bullet$, the \emph{$0$-coskeleton}, defined by taking $\cosk_0(X_0)_n$ to be the $n$-fold fiber
product of $X_0$ over $X$, with the necessary maps given by the various projections (resp.
diagonals) to (resp. from) $\cosk_0(X_0)_{n-1}$. The $0$-coskeleton has the following universal
property: If $Y_\bullet$ is a simplicial scheme with augmentation to $X$, then there is a
bifunctorial bijection $\hom_X(Y_0,X_0)\cong\hom_X(Y_\bullet, \cosk_0(X_0)_\bullet)$. In
particular, if $X_\bullet$ is a simplicial scheme, then there is a unique simplicial
$X$-morphism $\gamma:X_\bullet\rightarrow
\cosk_0(X_0)_\bullet$, with $\gamma_0=\id_{X_0}$. For the general construction see, e.g., 
\cite[5.1.1]{HodgeIII}.
\begin{lemma}\label{cosklemma}
	If $\tau:X_\bullet\rightarrow X$ is an augmented simplicial $k$-scheme such that
	$X_n$ is separated and of finite type over $k$ for all $n$, and $\gamma:X_\bullet\rightarrow
	\cosk_0(X_0)_\bullet$ the morphism of augmented simplicial schemes such that
	$\tau_0=\id_{X_0}$, then the kernel of the induced morphism $\Pic(
	(\cosk_0(X_0)_\bullet)\rightarrow \Pic X_\bullet$ is finitely generated.
\end{lemma}
\begin{proof}
        We have the following situation:
	\[\xymatrix{
	\gamma_\bullet:X_\bullet\ar[r] \ar@{}[d]|\vdots&
	\cosk_0(X)_\bullet\ar@{}[d]|\vdots\\
	X_1\ar@<1ex>[d]^{\delta_i}\ar[d] \ar[r]^{\gamma_1}&X_0'
	\ar@<1ex>[d]^{p_i}\ar[d] \\
	X_0\ar@<1ex>[u]\ar@{=}[r]^{\gamma_0=\id}\ar[d]&X_0\ar@<1ex>[u]\ar[d]\\
	X\ar@{=}[r]&X
	}\]
	where $X_0':=X_0\times_X X_0$ and $p_i$ the projection to the $i$-th factor.
	
	If $(L,\alpha:p_1^*L\stackrel{\sim}{\rightarrow}
	p_2^*L)\in \Pic( \cosk_0(X_0)_\bullet)$ pulls back to the trivial element in $\Pic(X_\bullet)$,
	then there is some isomorphism $\beta:L\rightarrow \mathcal{O}_{X_0}$, such that
	the diagram
	\[\xymatrix{
	\delta_0^*L\ar[r]^{\gamma_1^*\alpha}\ar[d]_{\delta_0^*\beta}&
	\delta_1^*L\ar[d]^{\delta_1^*\beta}\\
	\mathcal{O}_{X_1}\ar@{=}[r]^\id&\mathcal{O}_{X_1}}\]
	commutes, and $(p_2^*\beta) \alpha (p_1^*\beta)^{-1}$ is an automorphism of
	$\mathcal{O}_{X_0'}$, pulling back the identity on $X_1$. Hence
	$(p_2^*\beta)\alpha(p_1^*\beta)^{-1}$ is an element of
	\[\ker(\gamma_1^*:\Gamma(X_0',\mathcal{O}_{X_0'}^\times)\longrightarrow
	\Gamma(X_1,\mathcal{O}_{X_1}^\times)).\] 
	Note that $(p_2^*\beta) \alpha (p_1^*\beta)^{-1}=1$ if,
	and only if, $(L,\alpha)$ is already trivial in $\Pic(\cosk_0(X_0)_\bullet)$. 

	Replacing  $\beta$ by
	$\beta \lambda$ for some $\lambda\in
	\ker(\delta_0^*-\delta_1^*:\Gamma(X_0,\mathcal{O}_{X_0}^\times)\rightarrow
	\Gamma(X_1,\mathcal{O}_{X_1}^\times))$ gives a new trivialization
	$\gamma^*(L,\alpha)\cong(\mathcal{O}_{X_0},\id)$, and any trivialization can be
	reached like this (trivializations of the line bundle $L$ are a $\G_m$-torsor, and
	to get a trivialization of the \emph{pair} $\gamma^*(L,\alpha)=(L, \gamma_1^*\alpha)$, the condition that
	$1=(\delta_1^*\lambda)^{-1}(\delta_0^*\lambda)$ is necessary and sufficient). Next, observe that
	$p_1^*-p_2^*$ (or rather $p_1^*/p_2^*$) induces a map
	\[\ker(\Gamma(X_0,\mathcal{O}_{X_0}^\times)\stackrel{\delta_0^*-\delta_1^*}{\longrightarrow}
	\Gamma(X_1,\mathcal{O}_{X_1}^\times))\rightarrow\ker(\Gamma(X_0',\mathcal{O}_{X_0'}^\times)\stackrel{\gamma_1^*}{\longrightarrow}\Gamma(X_1,\mathcal{O}_{X_1}^\times)).\]
	Putting all of this
	together, we see that we obtain  an injective map
	\begin{multline*}
		\ker\left(\Pic( (\cosk_0 X_0\right)_\bullet)\rightarrow
	\Pic(X_\bullet))\\\longrightarrow
	\frac{\ker(\Gamma(X_0',\mathcal{O}_{X_0'}^\times)\stackrel{\gamma_1^*}{\longrightarrow}\Gamma(X_1,\mathcal{O}_{X_1}^\times))}{(p_1-p_2)^*(\ker(\Gamma(X_0,\mathcal{O}_{X_0}^\times)\stackrel{\delta_0^*-\delta_1^*}{\longrightarrow}
	\Gamma(X_1,\mathcal{O}_{X_1}^\times))}.
	\end{multline*}
	This implies
	that $\ker(\Pic(\cosk_0(X_0)_\bullet)\rightarrow \Pic(X_\bullet))$ is finitely
	generated. In fact, pulling back units from $k^\times$ by $\gamma_1$ is
	injective, as $k\rightarrow \Gamma(X_1,\mathcal{O}_{X_1})$ is injective. Thus
	$\ker(\Gamma(X_0',\mathcal{O}_{X_0'}^\times)\stackrel{\gamma_1^*}{\longrightarrow}\Gamma(X_1,\mathcal{O}_{X_1}^\times))\hookrightarrow
	\Gamma(X_0',\mathcal{O}_{X_0'}^\times)/k^\times$, which is a finitely generated
	abelian group. To see this we use the
	separatedness of $X_0$ to ensure the existence of a Nagata compactification of
	$X'_0$, so that we can apply, e.g., \cite[Lemme 1]{Kahn/GroupeDesClasses}.
\end{proof}

\begin{proposition}\label{ufg} Let $U$ be a regular, connected $k$-scheme of finite type, and
	$\tau: U_\bullet\rightarrow U$ a smooth, proper
	hypercovering such that $\tau_0:U_0\rightarrow U$ is an alteration (i.e.~proper,
	surjective and \emph{generically finite}) and $U_0$ is connected. Then the kernel
	of $\tau^*:\Pic U \rightarrow \Pic U_\bullet$ is finitely generated. In
	particular, $\Pic U_\bullet$ is finitely
	generated, then so is $\Pic U$.
\end{proposition}
\begin{proof}
	If
	$V\subset U$ is the biggest open subset of $U$ such that $\tau_0$ restricted to $V_0:=\tau_0^{-1}(V)$ is
	flat, then  $V\neq \emptyset$, and the complement $U\setminus V$ has codimension $\geq 2$. In fact, as
	$U_0\rightarrow U$ is surjective, for any $\eta$ mapping to a codimension $1$
	point $\xi\in U$, the morphism
	$\mathcal{O}_{U,\xi}\rightarrow \mathcal{O}_{U_0,\eta}$ is injective, so
	$\mathcal{O}_{U_0,\eta}$ is a torsion free $\mathcal{O}_{U,\xi}$-module. But as $U$
	is regular, $\mathcal{O}_{U,\xi}$ is a discrete
	valuation ring, so $\tau_0$ is flat at $\eta$, and $\xi\in V$.
	Thus $\Pic(U)=\Pic(V)$, and $\tau_0|_{V_0}$ is faithfully flat.  
	
	Giving an element of $\Pic( \cosk_0(U_0)_\bullet)$ is the same thing as giving an
	(isomorphism class of) a pair $(L,\alpha)$ with $L$ a line bundle on $U_0$ and
	$\alpha$ a descent datum of $L$ relative to $U$.

	Finally, we see that if a line bundle
	$L$ on $U$ pulls back to the trivial descent datum, then restricting it to $V_0$ and using
	faithful flatness shows that $L|_{V}$ is trivial, so $L$ is trivial, as
	$U\setminus V$ has codimension $\geq 2$. Hence $\Pic
	U\rightarrow \Pic (\cosk_0(U_0)_\bullet)$ is injective, and by Lemma
	\ref{cosklemma} this implies that the kernel of  $\tau^*:\Pic U\rightarrow \Pic
	U_{\bullet}$ is finitely generated.
\end{proof}

\begin{proposition}\label{surjection} Let $j:U_\bullet \rightarrow X_\bullet$ be a morphism of $k$-simplicial
	schemes, such that 
	\begin{enumerate}
		\item $X_p$ is regular and proper over $k$ for every $p$,
		\item $j_p:U_p\hookrightarrow X_p$ is an open immersion with dense image,
		\item \label{assumption3}the face maps $X_{i+1}\rightarrow X_i$ map $X_{i+1}\setminus
			U_{i+1}$ to $X_i\setminus U_i$.
	\end{enumerate}
	Then the cokernel of the induced map $j^*:\Pic X_\bullet\rightarrow \Pic U_\bullet$ is finitely generated.
\end{proposition}
\begin{proof}
	Let $K_X^i:=H^0(X_i,\mathcal{O}_{X_i}^\times)$ and make it into a complex of abelian
	groups
	$(K_X,\bar{\delta})$ via $\bar{\delta}_i:=\sum_{\ell=0}^{i+1} (-1)^\ell\delta_\ell^*: K_X^i\rightarrow K_X^{i+1}$.
	Define $K_U$ in the analogous fashion (where, to simplify notation, we write
	$\delta_i$ for the faces of $X_\bullet$ and for the faces of $U_\bullet$).
	Note that the complexes $K_X$ and $K_U$ have finitely generated cohomology groups:
	For even $i>0$ we
	have $k^\times\subset \image(\bar{\delta}_{i-1})$, so
	$H^i(K_X)=\ker(\bar{\delta}_i)/\image(\bar{\delta}_{i-1})$ is a
	subquotient of $\Gamma(X_i,\mathcal{O}_{X_i}^\times)/k^\times$, which is finitely
	generated, see e.g. \cite[Lemme 1]{Kahn/GroupeDesClasses}. The same argument holds for $K_U$.
	For odd $i$, we have $k^\times\cap \ker(\bar{\delta}_i)=1$, so
	$\ker(\bar{\delta}_i)\hookrightarrow
	\Gamma(X_i,\mathcal{O}_{X_i}^\times)/k^\times$, and hence $H^i(K_X)$ and
	$H^i(K_U)$ are finitely generated as well. 

	The morphism $j$ induces a morphism of spectral sequences
	\[\xymatrix{E_{1,X}^{p,q}=H^q(X_p,\mathcal{O}_{X_p}^\times)\ar@{=>}[r]\ar[d]&
	\mathbb{H}^{p+q}(X_\bullet,\mathcal{O}_{X_\bullet}^\times)\ar[d]\\
	E^{p,q}_{1,U}=H^q(U_p,\mathcal{O}_{U_p}^\times)\ar@{=>}[r]&
	\mathbb{H}^{p+q}(U_\bullet,\mathcal{O}_{U_\bullet}^\times),}\]
	(note that $K_X=E_{1,X}^{\cdot,0}$, and similarly for $K_U$) from which we obtain the morphism of short exact sequences
	\[\xymatrix{
	0\ar[r]& E_{2,X}^{1,0}=H^1(K_X)\ar[r]\ar[d]&\Pic(X_\bullet)\ar[d]^{j^*}\ar[r]&\ker(d_{2,X})\ar[d]\ar[r]&0\\
	0\ar[r]& E_{2,U}^{1,0}=H^1(K_U)\ar[r]&\Pic(U_\bullet)\ar[r]&\ker(d_{2,U})\ar[r]&0,\\
	}\]
	where $d_{2,X}$ is the differential 
	\[E_{2,X}^{0,1}=\ker(\Pic X_0\stackrel{\delta_0^*-\delta_1^*}{\longrightarrow} \Pic X_1)\longrightarrow
	H^2(K_X)=E_{2,X}^{0,2},\]
	and similarly for $d_{2,U}$. Since $H^1(K_U)$ is finitely generated, $\coker(H^1(K_X)\rightarrow
	H^1(K_U))$ is also finitely generated, so to finish the proof of the proposition, it
	remains to show that $\coker(\ker(d_{2,X})\rightarrow \ker(d_{2,U}))$ is finitely
	generated.

	Consider the diagram
	\[\xymatrix{
	0\ar[r]&\ker(d_{2,X})\ar[r]\ar[d]& \ker(\Pic
	X_0\stackrel{\delta_0^*-\delta_1^*}{\longrightarrow} \Pic
	X_1)\ar[r]^>>>>{d_{2,X}}\ar[d]_{\phi_0}& \image(d_{2,X})\ar[r]\ar[d]&0\\
	0\ar[r]&\ker(d_{2,U})\ar[r]&\ker(\Pic
	U_0\stackrel{\delta_0^*-\delta_1^*}{\longrightarrow}\Pic
	U_1)\ar[r]^>>>>{d_{2,U}}&\image(d_{2,U})\ar[r]&0.}\]
	As $\image(d_{2,X})\subset H^2(K_X)$, we know that $\ker(\image(d_{2,X})\rightarrow
	\image(d_{2,U}))$ is finitely generated, so by the Snake Lemma, to finish the proof it suffices to
	show that the middle vertical map $\phi_0:E_{2,X}^{0,1}\rightarrow E_{2,U}^{0,1}$,
	$\phi_0(L)=L|_{U_0}$ from above has a finitely generated cokernel.
	
	By our regularity assumptions, we have for each $i$ an exact sequence
	\[0\longrightarrow \mathbb{Y}_{i}\longrightarrow \Pic X_i \longrightarrow \Pic
	U_i\longrightarrow 0,\]
	where $\mathbb{Y}_i$ is the subgroup of $\Pic X_i$ generated by the classes of the
	(finite number of) codimension $1$ points of $X_i\setminus U_i$. In particular, this induces a map
	\[\ker( \bar{\delta}_i^*:\Pic U_i\rightarrow\Pic U_{i+1})\longrightarrow
	\mathbb{Y}_{i+1}/\bar{\delta}_i^*
	\mathbb{Y}_i,\]
	where $\bar{\delta}_i^*=\sum_{\ell=0}^{i+1}(-1)^{\ell}\delta_\ell^*$.
	Indeed, we may extend $L\in \ker(\bar{\delta}_i^*:\Pic U_i\rightarrow \Pic U_{i+1})$ to some
	$\tilde{L}\in \Pic X_i$, and map it to
	$\Pic X_{i+1}$, where it has support contained in $X_{i+1}\setminus
	U_{i+1}$, i.e.~it is mapped to $\mathbb{Y}_{i+1}$. To get a well-defined map on
	$\Pic U_i$, we
	have to account for the choice of the extension of $L$ to $X_i$, that is we have to divide out by the
	image of $\mathbb{Y}_{i}$ under $\bar{\delta}_i^*$ which is contained in
	$\mathbb{Y}_{i+1}$ by assumption \ref{assumption3}.

	Next, we show that the kernel of this map is precisely the image of the restriction
	\[\phi_i:\ker(\Pic X_i\rightarrow \Pic X_{i+1})\longrightarrow\ker(\Pic
	U_i\rightarrow \Pic U_{i+1}).\] 
	If
	$L\in \ker(\Pic U_i\rightarrow \Pic U_{i+1})$ maps to $\bar{\delta}_i^*M$, for some
	$M\in \mathbb{Y}_i$, then there is some extension $\tilde{L}$ of $L$ to $X_i$,
	such that $\bar{\delta}_i^*(\tilde{L}\otimes M^{-1})\cong \mathcal{O}_{X_{i+1}}$, so
	$\tilde{L}\otimes M^{-1}\in \ker(\Pic X_i\rightarrow \Pic X_{i+1})$. This shows
	$L\cong\phi_i(\tilde{L}\otimes M^{-1})$, as $M$ is supported on $X_i\setminus U_i$.
	Conversely, if some $L\in \ker(\Pic U_i\rightarrow \Pic U_{i+1})$ can be extended
	to $\tilde{L}\in \ker(\Pic X_i\rightarrow \Pic X_{i+1})$, then by definition $L$
	maps to $0$ in $\mathbb{Y}_{i+1}/\bar{\delta_i}^*\mathbb{Y}_{i}$.

	This finishes the proof: Specializing the last calculation to $i=1$, we see that
	$\coker(\phi_0)$ can be embedded into the finitely generated group
	$\mathbb{Y}_1/(\delta_0^*-\delta_1^*)\mathbb{Y}_0$.
\end{proof}

\section{The Picard Group of Simply Connected Varieties}\label{section:simplyconnectedvarieties}
By a simply connected scheme we mean an irreducible scheme $X$ such that
$\pi_1^{\et}(X,\bar{x})=1$
for some (or any) geometric point $\bar{x}$ of $X$. Often we will suppress notation of base
points and write $\pi_1$ for $\pi_1^{\et}$. If $X$ is a $k$-scheme, for some field $k$,
then $k$ is necessarily algebraically closed. We will mostly be interested in the case
$\Char(k)=p>0$.

\begin{proposition}\label{simplicialpicfg}
	If $X$ is a normal, proper, connected $k$-scheme of finite type, such that
	$\pi_1(X)^{\ab,(\ell)}=1$ for some $\ell\neq p$, and
	$X_\bullet\rightarrow X$ a proper hypercovering with $X_0$ normal, and $X_n$
	reduced for all $n$, then
	$\NS(X_\bullet)=\Pic(X_\bullet)$. In particular, $\Pic(X_\bullet)$ is finitely
	generated.
\end{proposition}
\begin{proof}
	This is a consequence of cohomological descent: There is an isomorphism $0=H^1_{\et}(X,\mu_n)\cong \mathbb{H}^1(X_\bullet,
	\mu_{n,X_\bullet})$ (see e.g. \cite[Lemma 5.1.3]{Srinivas/Motives}), so
	$\Pic^0_{X_\bullet/k}(k)=\Pic^{0,\red}_{X_\bullet/k}(k)$ has no $\ell$-torsion, and thus is trivial, as $\Pic^{0,\red}_{X_\bullet/k}$ is semi-abelian by
	Theorem \ref{representabilityandfiniteness}. In fact, if  a semi-abelian variety has no
	$\ell$-torsion, then it is an abelian variety, as a nontrivial subtorus would have
	nontrivial $\ell$-torsion. But an abelian variety with trivial $\ell$-torsion is
	trivial. Hence
	$\NS(X_\bullet)=\Pic(X_\bullet)=\Pic_{X_\bullet/k}(k)$. This group is finitely
	generated by Proposition \ref{simplicialnsfg}.
\end{proof}

We are ready to prove the first main theorem.
\begin{proof}[{Proof of Theorem \ref{picfg}}.]
	By Nagata's theorem there exists a proper variety $X$
	admitting $U$ as a dense open subscheme, and since $U$ is normal we may assume $X$ to be
	normal.  By \cite{deJong/Alterations} there exists an augmented proper hypercovering
	$X_\bullet\rightarrow X$ with $X_n$ regular and proper over $k$, such that the part
	$Z_n$ of $X_n$ lying over
	$X\setminus U$ is a strict normal crossings divisor. Write $U_n:=X_n\setminus
	Z_n$. As $U$ is connected we can
	pick $X_\bullet$ such that $X_0$ and $U_0$ are connected. Also note that
	$\pi_1(U)^{\ab,(\ell)}$ surjects onto $\pi_1(X)^{\ab,(\ell)}$ (see, e.g.,  \cite[Prop.
	V.6.9]{SGA1}), so $\pi_1(X)^{\ab,(\ell)}=1$.  We have shown that $\Pic
	X_\bullet$ is finitely generated (Proposition \ref{simplicialpicfg}), that $\Pic
	X_\bullet$ maps to $\Pic U_\bullet$ with finitely generated cokernel (Proposition
	\ref{surjection}) and that this implies that $\Pic U$ is finitely generated
	(Proposition
	\ref{ufg}).
\end{proof}

\section{Stratified Line Bundles on Regular Simply Connected
Varieties}\label{section:stratified}
We continue to denote by $k$ an algebraically closed field of characteristic
$p\geq 0$.

Let $U$ be a $k$-scheme. Grothendieck defined in \cite[\S 16]{EGA4} the sheaf
$\mathcal{D}_{U/k}$ of differential operators of $U$ over $k$. In characteristic $0$,
$\mathcal{D}_{U/k}$ is (locally) the enveloping algebra of $\mathcal{O}_U$ and
the sheaf of derivations $\Der_{U/k}(\mathcal{O}_U,\mathcal{O}_U)$. This is false in positive characteristic. For
details, see e.g. \cite[Ch. 2]{BerthelotOgus/Crystalline}. 
\begin{definition}
	A $\mathcal{O}_U$-coherent module $E$ is called \emph{stratified bundle}, if $E$
	has a $\mathcal{D}_U$-action, compatible with its $\mathcal{O}_U$-structure. A
	\emph{horizontal morphism of stratified bundles} $E\rightarrow E'$ is a morphism of
	$\mathcal{O}_X$-modules, which is also a morphism of $\mathcal{D}_U$-modules. A
	stratified bundle is trivial, if it is isomorphic to $\mathcal{O}_U^n$ for some $n$, and if the $\mathcal{D}_U$-action is the sum of the canonical actions on
	$\mathcal{O}_{U}$.
\end{definition}
Note that if $U$ is regular, then a stratified bundle is automatically locally free, \cite[2.17]{BerthelotOgus/Crystalline}, which
justifies the name.

In positive characteristic, Katz gives a nice
description of stratified bundles.

\begin{theorem}[{Katz, \cite[Thm. 1.3]{Gieseker/FlatBundles}}]
	Let $k$ be an algebraically closed field of positive characteristic $p$. If $U$ is a regular, finite type $k$-scheme, then the category of stratified bundles on $U$ is equivalent to the category of
	sequences of pairs $(E_n, \sigma_n)_{n\in \N}$, where $E_n$ is a locally free
	sheaf of finite rank on $U$, and $\sigma_n$ an isomorphism $F^*E_{n+1}\rightarrow
	E_n$, with $F:X\rightarrow X$ the absolute Frobenius. 

	A morphism
	$(E_n,\sigma_n)\rightarrow (E'_n,\sigma'_n)$ in the latter category is given by a
	sequence
	morphisms $\phi_n:E_n\rightarrow E'_n$ compatible with the $\sigma_n,\sigma'_n$.

	A trivial stratified bundle is corresponds to a sequence of pairs
	$(\mathcal{O}_U,\id_{\mathcal{O}_U})_n$.
\end{theorem}

We will use this characterization to prove Theorem \ref{thm:nostratifiedlinebundles}, but
first we need a statement about global functions on simply connected schemes.

\begin{proposition}\label{globalunits}
	If $U$ is a connected normal $k$-scheme of finite type, such that the maximal
	abelian pro-$\ell$-quotient
	$\pi_1(U)^{\ab,(\ell)}$ is trivial for some $\ell\neq p$, then
	$H^0(U,\mathcal{O}_U^\times) = k^\times$. If $k$ has positive characteristic $p$, and
	$\pi_1(U)^{(p)}=1$, then $H^0(U,\mathcal{O}_U)=k$.
\end{proposition}
\begin{proof} The argument for the first assertion is due to H\'el\`ene Esnault.
	Assume $f\in H^0(U,\mathcal{O}_U^\times)\setminus k^\times$. Then $f$ induces a
	dominant morphism $f':U\rightarrow \mathbb{G}_{m,k}\cong
	\mathbb{A}^1_k\setminus\left\{ 0 \right\}$, as $f'$ is given by the map $k[x^{\pm 1}]\rightarrow H^0(U,\mathcal{O}_U)$,
	$x\mapsto f$, which is injective if, and only if, $f$ is transcendental over $k$.
	Thus $f'$ induces an \emph{open} morphism $\pi_1(U)\rightarrow
	\pi_1(\mathbb{G}_{m,k})$, see e.g. \cite[Lemma 4.2.10]{Stix/dissertation}. 
	But under our assumption, the maximal abelian pro-$\ell$-quotient of the
	image of this morphism is trivial, so the image of $\pi_1(U)$ cannot have finite index in the
	group $\pi_1(\mathbb{G}_{m,k})$, as in fact
	$\pi_1(\mathbb{G}_{m,k})^{(\ell)}\cong
	\widehat{\Z}^{(\ell)}=
	\Z_\ell$.

	For the second assertion, if $f\in H^0(U,\mathcal{O}_U)\setminus k$, then by the\
	same arguments as above, $f$ induces a dominant morphism $U\rightarrow
	\mathbb{A}_k^1$, and hence an open map $\pi_1(U)\rightarrow\pi_1(\mathbb{A}^1_k)$.
	For $k$ of positive characteristic it is known that $\pi_1(\mathbb{A}^1_k)$ has an
	infinite maximal pro-$p$-quotient (in fact it is a free pro-$p$-group of infinite rank: by
	\cite[1.4.3, 1.4.4]{Katz/LocalToGlobal} we have $H^2(\pi_1(U),\F_p)=0$, so $\pi_1(U)^{(p)}$
	is free pro-$p$ of rank $\dim_{\F_p}H^1(\A_k^1,\F_p)=\#k$). Thus the image of $\pi_1(U)$ in this group can only have finite index, if
	$\pi_1(U)^{(p)}\neq 1$.
\end{proof}
\begin{remark}\label{rem:affinevars}
	Proposition \ref{globalunits} gives a proof of the well known fact that over a
	field $k$ of positive characteristic, unlike in characteristic $0$, no affine $k$-scheme of positive dimension is simply connected.
\end{remark}

\begin{lemma}\label{isoclassesgeneral}Let $U$ be a connected normal $k$-scheme of
	finite type. If $\pi_1(U)^{\ab,(\ell)}=1$ for some prime $\ell\neq p$, then the isomorphism class of a stratified line
	bundle $L=\left( L_n,\sigma_n \right)_n$ is uniquely determined by the isomorphism
	classes of the $L_n$. 
\end{lemma}
\begin{proof}
	This follows from Proposition \ref{globalunits}, and the argument is essentially
	contained in the proof of \cite[Prop. 1.7]{Gieseker/FlatBundles}: Let $M:=(M_n,\tau_n)$ be
	a second stratified line bundle on $U$ and $u_n:L_n\rightarrow M_n$
	isomorphisms of $\mathcal{O}_{X}$-modules. We will construct an
	isomorphism of stratified line bundles $L\rightarrow M$. Consider the
	following diagram:
	\[\xymatrix{
	L_0\ar[d]_{u_0}^\cong\ar[r]^{\sigma_0} & F^*L_1\\
	M_0\ar[r]^{\tau_0}&F^*M_1}\]
	The automorphism $\lambda:=\tau_0u_0\sigma_0^{-1}F^*(u_1^{-1})$ of $F^*M_1$
	corresponds to a global unit $\lambda\in \Gamma(U,\mathcal{O}_{U}^\times)$.
	By Proposition \ref{globalunits}, $U$ has only constant global
	units, so there is a $p$-th root $\lambda^{1/p}$ of $\lambda$, which
	defines an automorphism of $M_1$ such that $F^*\lambda^{1/p}=\lambda$.
	Defining $f_0:=u_0$ and $f_1:=\lambda^{1/p}u_1$ gives the first two
	steps of defining an isomorphism of stratified bundles $f:L\rightarrow
	M$. We can continue this process.
\end{proof}

The following fact from group theory is elementary.
\begin{lemma}\label{grouplemma}
	Let $G$ be a finitely generated abelian group and $p$ a prime number. A nontrivial element $L\in G$ is
	infinitely $p$-divisible if and only if $L$ has finite order prime to
	$p$. 
\end{lemma}
%
%

We can now easily prove the main result of this section.

\begin{proof}[{Proof of Theorem \ref{thm:nostratifiedlinebundles}}.] This is an adaptation of
	an argument from the introduction of \cite{Esnault/Giesecker}. By Lemma \ref{isoclassesgeneral} we only need to
	show that the classes of $L_n$ in $\Pic U$ are all trivial. Note that
	$L_n$ is infinitely $p$-divisible in $\Pic U$ for all $n$. For 
	regular $U$ as in the assertion, we know from Theorem \ref{picfg} that $\Pic U$ is
	finitely generated, and hence by Lemma \ref{grouplemma} it follows
	that $L_n$ is torsion of order prime to $p$ in $\Pic U$. But Kummer theory shows
	that $\Pic U$ does not have nontrivial prime-to-$p$ torsion.
	
\end{proof}
\section{A Complex Analog}\label{section:complex}
In this section, we reproduce a complex analog of Theorem
\ref{thm:nostratifiedlinebundles} due to H\'el\`ene Esnault.

\begin{definition}
	Let $\Pic^\nabla(U)$ denote the group of isomorphism classes of pairs
	$(L,\nabla)$, where $L$ is an invertible $\mathcal{O}_U$-module and $\nabla$ an
	integrable connection on $L$. The group structure is given by tensor product
	connections.
\end{definition}
\begin{remark}
	For a regular scheme over a field of characteristic $0$, the notion of stratified
	bundles is equivalent to the notion of vector bundles with integrable connection.
\end{remark}

Let $\Omega_{U/C}^\times$ denote the complex of abelian sheaves in the Zariski topology
\[\mathcal{O}^\times_X\stackrel{\dlog}{\rightarrow} \Omega_{U/\C}^1\rightarrow
\Omega_{U/\C}^2\rightarrow \ldots,\]
and note that we obtain a short exact sequence of complexes
\[ 0\rightarrow \Omega^{\geq 1}_{U/k}\rightarrow \Omega^\times_{U/k}\rightarrow
\mathcal{O}_U^\times \rightarrow 0.\]
From this we obtain a homomorphism $\tilde{c}_1:\Pic U \rightarrow
\mathbb{H}^2(U,\Omega_{U/k}^{\geq 1})$, which can be described explicitly  as follows: Let
$L\in \Pic U$ be a line bundle. As $\Omega^{\geq 1}_{U/\C}$ is a complex of coherent
sheaves, we can compute its hypercohomology directly via \v{C}ech theory. Let
$\mathcal{U}=\{U_i\}_i$ be an open affine covering of $U$ trivializing $L$, such that
$L|_{U_i}=e_i\mathcal{O}_{U_i}$, and $e_i\xi_{ij}=e_j$ on $U_i\cap U_j=:U_{ij}$, with $\xi_{ij}\in
\mathcal{O}_{U_{ij}}^\times(U_{ij})$. If
$\mathcal{C}^p(\mathcal{U},\Omega^q_{U/\C})=\prod_{i_1,\ldots, i_p}
\Omega^q_{U/\C}(U_{i_1\cdots i_p})$ denotes the associated \v{C}ech double complex, then
$\tilde{c}_1$ is induced by \[L\mapsto ((\dlog \xi_{ij})_{ij},0)\in
\mathcal{C}^1(\mathcal{U},\Omega^1_{U/\C})\oplus
\mathcal{C}^0(\mathcal{U},\Omega^2_{U/\C})=\Tot^2(\mathcal{C}^\bullet(\mathcal{U},\Omega^{\geq
1}_{U/\C})),\] where $\Tot$ denotes the total complex of a double complex. From this
description it follows that $\tilde{c}_1$, composed with the natural map
$\mathbb{H}^2(U,\Omega_{U/k}^{\geq 1})\rightarrow
\mathbb{H}^2(U,\Omega^\bullet_{U/k})=H^2(U,\C)$, is the usual complex first Chern class.

\begin{proposition}\label{prop:exactcomplexsequence}
	If $U$ is a regular, finite type $\C$-scheme, then the natural morphism
	$\Pic^\nabla(U)\rightarrow \Pic(U)$, $(L,\nabla)\mapsto L$, fits in a short exact
	sequence
	\[0\rightarrow H^0(U,\Omega^1_{U/\C,\cls})/\dlog
	H^0(U,\mathcal{O}_U^\times)\rightarrow
	\Pic^\nabla(U)\rightarrow \ker(\tilde{c}_1)\rightarrow 0,\]
	where $\Omega_{U/\C,\cls}^1$ denotes the closed $1$-forms.
\end{proposition}
\begin{proof}
	Consider the map $\varphi: H^0(U,\Omega_{U/\C,\cls}^1)\rightarrow \Pic^\nabla(U)$ given by
	$\omega\mapsto (\mathcal{O}_U,d+\omega)$. Every integrable connection on $\mathcal{O}_U$ is
	given by $d+\omega$ for a closed $1$-form $\omega$, so $H^0(U,\Omega_{U/\C,\cls})$ maps onto
	the kernel of $\Pic^\nabla(U)\rightarrow \Pic U$.
	The kernel of $\varphi$ consists precisely of those $1$-forms
	$\omega$ such that there is a $\lambda \in \mathcal{O}_U^\times$ with
	$d(1)=0=d\lambda +\lambda \omega$, i.e.~$\ker(\varphi)=\dlog \mathcal{O}_U^\times$. We
	have proven that the sequence
	\[
	0\rightarrow H^0(U,\Omega^1_{U/\C,\cls})/\dlog
	H^0(U,\mathcal{O}_U^\times)\rightarrow
	\Pic^\nabla(U)\rightarrow \Pic(U)\]	
	is exact. 

	It remains to determine the image of $\Pic^\nabla(U)\rightarrow \Pic(U)$. Let
	$(L,\nabla)$ be a line bundle with integrable connection. Let $\mathcal{U}=\{U_i\}$ be an open affine covering of $U$
	trivializing $L$, such that $L|_{U_i}=e_i$, and $e_i\xi_{ij}=e_j$ on $U_i\cap
	U_j=:U_{ij}$, with $\xi_{ij}\in \mathcal{O}_{U_{ij}}^\times$. 
	We have seen that $\tilde{c}_1$ is given by
	\[L\mapsto ((\dlog \xi_{ij})_{ij},0)\in
	\mathcal{C}^1(\mathcal{U},\Omega^1_{U/\C})\oplus
	\mathcal{C}^0(\mathcal{U},\Omega^2_{U/\C})=\Tot^2(\mathcal{C}^\bullet(\mathcal{U},\Omega^{\geq
	1}_{U/\C})),\]
	Now if we write $\omega_i:=\nabla(e_i)\in \Omega_{U/\C,\cls}(U_i)$, then we get on
	$U_{ij}$
	\[ \omega_i\xi_{ij}e_j=\omega_i e_i=\nabla(e_i)=\nabla(\xi_je_j)=(\xi_{ij}w_j +d\xi_{ij})e_j\]
	and thus $\omega_i-\omega_j|_{U_{ij}}=\dlog(\xi_{ij})$, which means that
	$(L,\nabla)$ maps to $\ker(\tilde{c}_1)$.

	Conversely, if $L$ is a line bundle with $\tilde{c}_1(L)=0$, then there is an
	open affine covering $\mathcal{U}=\{U_i\}$ trivializing $L$ with
	$L|_{U_i}=e_i \mathcal{O}_U$, and $e_i\xi_{ij}=e_j$ on $U_{ij}$, such that
	$\tilde{c}_1(L)=( (\dlog \xi_{ij})_{ij}, 0)$ can be represented as $( (
	\omega_{i}|_{U_{ij}}-\omega_j|_{U_{ij}})_{ij},
	0)\in \mathcal{C}^1(\Omega^1_{U/\C})\oplus
	\mathcal{C}^0(\Omega^2_{U/\C})$, with $\omega_{i}\in \Omega_{U/\C}^1(U_i)$.
	Defining $\nabla(e_i)=\omega_i$ then defines an integrable connection on $L$, so
	the image of $\Pic^\nabla(U)\rightarrow \Pic(U)$ is precisely $\ker(\tilde{c}_1)$.
\end{proof}

\begin{theorem}\label{thm:complexanalogue}
	If $U$ is a regular, finite type $\C$-scheme, such that $\pi_1(U)^{\et,\ab}=0$, then
	$H^0(U,\Omega^1_{U/\C,\cls})\cong \Pic^\nabla(U)$, via the map $\omega\mapsto
	(\mathcal{O}_U,d+\omega)$.
\end{theorem}
\begin{proof}
	By Proposition \ref{globalunits} we have $H^0(U,\mathcal{O}_U^\times)=\C^\times$,
	so Proposition \ref{prop:exactcomplexsequence} shows that it suffices to prove
	that $\ker(\tilde{c}_1)=0$, and for this it is enough to prove that $c_1:\Pic
	U\rightarrow H^2(U,\C)$ is injective. 
	
	Let $L$ be a line bundle with
	$c_1(L)=0$. Let $X$ be a regular compactification of $U$, such that
	$Y:=X\setminus U$ is a normal crossings divisor, and choose an extension
	$\overline{L}\in \Pic X$ of $L$. Note that $\pi_1(X)^{\ab}=1$, and that this
	implies that
	$c_1:\Pic(X)\hookrightarrow H^2(X,\C)$ is injective, as $\Pic(X)=\NS(X)$ is a
	free abelian group of finite rank. Thus $c_1(\overline{L})\in \ker
	(H^2(X,\C)\rightarrow H^2(U,\C))$. Since $\pi_1(U)^{\ab}$ is trivial, we have
	$\ker(H^2(X,\C)\rightarrow H^2(U,\C))=H^2_Y(X,\C)$, and by purity $H^2_Y(X,\C)\cong
	\bigoplus_i c_1(Y_i)\C$ (see \cite[Cycle, 2.2.6, 2.2.8, 2.1.4]{SGA45}), if $Y_1,\ldots, Y_n$ are the regular components of $Y$. This implies that there is some
	$M\in \Pic(X)$, such that $M=\sum_i a_i Y_i$, with $c_1(\overline{L}\otimes M)=0$,
	so $M=\overline{L}^{-1}$. But $M|_U=\mathcal{O}_U$, so $L=\mathcal{O}_U$.
\end{proof}
\begin{remark}\label{rem:end}
	The difference between Theorem \ref{thm:nostratifiedlinebundles} and Theorem
	\ref{thm:complexanalogue} is related to the failure of the second assertion
	of Proposition \ref{globalunits} in characteristic $0$: There are simply connected
	complex varieties $U$ with nonconstant global functions, and on such $U$ the
	conclusion of Theorem \ref{thm:nostratifiedlinebundles} fails, as 
	$H^0(U,\Omega^1_{U/\C,\cls})\neq 0$. 

	On the other hand, from Deligne's Riemann-Hilbert
	correspondence, we know that there are no nontrivial \emph{regular singular}
	stratified bundles on $U$.
\end{remark}

%
\bibliographystyle{amsalpha}
\bibliography{alggeo}

\end{document}